\documentclass[12pt,reqno]{amsart}
\usepackage{amsmath, amssymb, amsthm}
\usepackage{url}
\usepackage[breaklinks]{hyperref}

\renewcommand{\Im}{\operatorname{Im}}

\renewcommand{\Re}{\operatorname{Re}}
\renewcommand{\Im}{\operatorname{Im}}

\renewcommand{\(}{\left\(}
\renewcommand{\)}{\right\)}
\renewcommand{\[}{\left\[}
\renewcommand{\]}{\right\]}

\numberwithin{equation}{section}
 \theoremstyle{plain}
\newtheorem{theorem}{Theorem}[section]
\newtheorem{lemma}[theorem]{Lemma}
\newtheorem{remark}[]{Remark}

\newtheorem{conjecture}[theorem]{Conjecture}

   \makeatletter
\def\proof{\@ifnextchar[{\@oproof}{\@nproof}}
\def\@oproof[#1][#2]{\trivlist\item[\hskip\labelsep\textit{#2 Proof of\
#1.}~]\ignorespaces}
\def\@nproof{\trivlist\item[\hskip\labelsep\textit{Proof.}~]\ignorespaces}

\makeatother

\begin{document}
\title[Equivalent criteria for the Generalized Riemann hypothesis]{Hardy-Littlewood-Riesz type equivalent criteria for the Generalized Riemann hypothesis}

\author{Meghali Garg}
\address{Meghali Garg \\ Department of Mathematics \\
Indian Institute of Technology Indore \\
Simrol,  Indore,  Madhya Pradesh 453552, India.} 
\email{meghaligarg.2216@gmail.com,   phd2001241005@iiti.ac.in}

 \author{Bibekananda Maji}
\address{Bibekananda Maji\\ Department of Mathematics \\
Indian Institute of Technology Indore \\
Simrol,  Indore,  Madhya Pradesh 453552, India.} 
\email{bibek10iitb@gmail.com,  bibekanandamaji@iiti.ac.in}

\thanks{2010 \textit{Mathematics Subject Classification.} Primary 11M06; Secondary 11M26 .\\
\textit{Keywords and phrases.} Riemann zeta function,  Dirichlet $L$-function,  Non-trivial zeros, Generalized Riemann hypothesis,  Equivalent criteria. }

\maketitle

\begin{abstract}
In the present paper,  we prove that the generalized Riemann hypothesis for the Dirichlet $L$-function $L(s,\chi)$ is equivalent to the following bound: Let $k \geq 1$ and $\ell$ be positive real numbers.  For any $\epsilon >0$,  we have
\begin{align*}
\sum_{n=1}^{\infty} \frac{\chi(n) \mu(n)}{n^{k}} \exp \left(- \frac{ x}{n^{\ell}}\right) = O_{\epsilon,k,\ell} \bigg(x^{-\frac{k}{\ell}+\frac{1}{2 \ell} + \epsilon }\bigg), \quad \mathrm{as}\,\, x \rightarrow \infty,
\end{align*}
where $\chi$ is a primitive Dirichlet character modulo $q$,  and  $\mu(n)$ denotes the M\"{o}bius function.  This bound generalizes the previous bounds given by Riesz,  and Hardy-Littlewood. 
\end{abstract}

\section{introduction}

The inception of the theory of the Riemann zeta function $\zeta(s)$,  is due to the German mathematician Bernhard Riemann in 1859.  This theory has played a crucial  role in development of the analytic number theory.  In his seminal paper \cite{Rie59},  Riemann showed that $\zeta(s)$ satisfies a beautiful symmetric functional equation and it can be analytically continued to the whole complex plane except at $s=1$.  In the same paper,  he conjectured that all the non trivial zeroes of $\zeta(s)$ lie on the critical line $\Re(s) = \frac{1}{2}$.  This conjecture is popularly known as the Riemann hypothesis and has been unsolved for the last 162 years.  Over the years,  mathematicians gave many useful equivalent criteria for the Riemann hypothesis while trying to prove it.  One of the notable equivalent criteria is due to Littlewood \cite[p.~370]{Tit},  which states that the Riemann hypothesis is equivalent to the following bound: 
for any $\epsilon >0$, 
\begin{align}\label{summatory_Mobius}
\sum_{1 \leq n \leq x} \mu(n) = O_{\epsilon}\left(x^{\frac{1}{2}+\epsilon}\right),  \quad \textrm{as} \,\,  x \rightarrow \infty,
\end{align}
where $\mu(n)$ denotes the M\"{o}bius function.  In 1916,  Riesz \cite{Riesz} found that the Riemann hypothesis is equivalent to the following bound for an infinite series associated to $\mu(n)$,  
\begin{equation}\label{Riesz}
P_2(x):=  \sum_{n=1}^\infty \frac{\mu(n)}{n^2} \exp\left( - \frac{x}{n^2} \right) = O_{\epsilon} \left( x^{-\frac{3}{4} + \epsilon} \right), \quad {\rm as}\,\,  x \rightarrow \infty,
\end{equation}
for any positive $\epsilon$.  Inspired from the work of Riesz,  Hardy and Littlewood \cite[p.~161]{HL-1916}  established another equivalent criterion,  namely,  the Riemann hypothesis is equivalent to the fact that
\begin{align}\label{Riesz type_Hardy_Littlewood}
P_1(x):= \sum_{n=1}^\infty  \frac{\mu(n)}{n} \exp\left({-\frac{x}{n^2}}\right) = O_{\epsilon}\left( x^{-\frac{1}{4}+ \epsilon } \right), \quad \mathrm{as}\,\, x \rightarrow \infty.  
\end{align} 
Hardy and Littlewood \cite[p.~156, Section 2.5]{HL-1916} obtained the bound \eqref{Riesz type_Hardy_Littlewood} while correcting   an  identity from the second notebook of Ramanujan  \cite[p.~312]{Rama_2nd_Notebook},  \cite[Equation (37.3), p.~470]{BCB-V}.  
Mainly,  they showed that,  for any positive real number $x$,  we have 
 \begin{align}\label{Hardy-Littlewood}
 \sum_{n=1}^{\infty} \frac{\mu(n)}{n} \exp\left({-\frac{x}{n^2}}\right) = \sqrt{\frac{\pi}{x}} \sum_{n=1}^{\infty} \frac{\mu(n)}{n} \exp\left(- \frac{\pi^2}{n^2 x} \right)- \frac{1}{2 \sqrt{\pi}} \sum_{\rho} \left(  \frac{\pi}{\sqrt{x}} \right)^{\rho}  \frac{ \Gamma\left(\frac{1-\rho}{2} \right) }{\zeta'(\rho)},
 \end{align}
 where the right hand side sum over $\rho$ runs through the non-trivial zeros of  $\zeta(s)$, which are assumed to be simple.  Substituting $\alpha \beta = \pi$ and  $x$ by $\alpha^2$,  the above identity takes the following shape:
\begin{align}\label{Rama_Hardy_Little}
\sqrt{\alpha} \sum_{n=1}^{\infty} \frac{\mu(n)}{n} \exp\left({-\left(\frac{\alpha}{n}\right)^2}\right) - \sqrt{\beta} \sum_{n=1}^{\infty} \frac{\mu(n)}{n} \exp\left({-\left(\frac{\beta}{n}\right)^2}\right) = -\frac{1}{2\sqrt{\beta}} \sum_{\rho} \frac{ \Gamma\left(\frac{1-\rho}{2} \right) \beta^\rho}{\zeta'(\rho)}.
\end{align}
The convergence of the right hand side infinite series over $\rho$ is quite intricate.  
Hardy and Littlewood \cite[p.~159, Section 2.5]{HL-1916}   and later Titchmarsh \cite[p.~220]{Tit} explained the convergence of this series assuming bracketing conditions of the ordinates of the non-trivial zeros of $\zeta(s)$.  This means, if $\rho_1$ and $\rho_2$ are two non-trivial zeros of $\zeta(s)$ such that for some positive constant $C$, 
\begin{align}\label{bracketing}
|\Im(\rho_1) - \Im(\rho_2)| < \exp \left( -\frac{C \Im(\rho_1)}{\log(\Im(\rho_1) )} \right) + \exp \left( -\frac{C \Im(\rho_2)}{\log(\Im(\rho_2) )} \right),
\end{align}
then the terms corresponding to these zeros will be considered in the same bracket.  Moreover,  Hardy and Littlewood also pointed out that it is highly possible that the infinite series over $\rho$ will be rapidly convergent but they were unable to prove it even after assuming Riemann hypothesis.  
The delicacy of convergence of such kind of series where $\zeta'(\rho)$ is present in the denominator has been explained in more details in \cite{JMS21}.  

The identity \eqref{Rama_Hardy_Little} has attracted the attention of many mathematicians over the years.  Readers are encouraged to see
 Berndt \cite[p.~470]{BCB-V}, Paris and Kaminski \cite[p.~143]{PK}, and Titchmarch \cite[p.~219]{Tit} for more information related to the identity \eqref{Rama_Hardy_Little}.  Interestingly, 
 Bhaskaran \cite{bhas} found a connection between Fourier reciprocity and Wiener's Tauberian theory.  Ramanujan himself indicated a nice generalization of \eqref{Rama_Hardy_Little} associated to a pair of reciprocal functions,  and this was worked out by Hardy and Littlewood \cite[p.~160, Section 2.5]{HL-1916}.  
  An elegant character analogue of \eqref{Rama_Hardy_Little} was established by Dixit \cite{dixit12} in 2012.  Later,  Dixit,  Roy and Zaharescu  \cite{DRZ-character} studied a one-variable generalization of \eqref{Rama_Hardy_Little} and derived equivalent criteria for the Riemann hypothesis, and also for the generalized Riemann hypothesis.   They \cite{DRZ}  also found an identity analogous to \eqref{Rama_Hardy_Little} for $L$-functions associated to Hecke eigenforms.  Last year,  Dixit,  Gupta and Vatwani \cite{DGV21}  obtained a generalization for the Dedekind zeta function and also established an equivalent criterion for the extended Riemann hypothesis.  A few months back,  Banerjee and Kumar \cite{BK21} derived an analogous identity corresponding to $L$-functions associated to the primitive cusp forms over $\Gamma_0(N)$ and found Riesz-Hardy-Littlewood type equivalent criteria for the corresponding $L$-function.  Interested readers can also see the work of B\'{a}ez-Duarte \cite{Baez1},  where a sequential Riesz-like criteria for the Riemann hypothesis was presented.  This phenomenon was further studied by Cislo and Wolf \cite{CW08}.

 Recently,   Agarwal and the authors \cite{AGM}  found a quite surprising one variable generalization of equation  \eqref{Hardy-Littlewood} by introducing an extra variable $k$.  The following  identity has been established: 
 
\emph {Let $k \geq 1$ be a real number and $x$ be a positive real number.  Under the assumption of the simplicity hypothesis of the non-trivial zeros of $\zeta(s)$,  we have
\begin{align}\label{AGM}
    \sum_{n=1}^{\infty} \frac{\mu(n)}{n^k} \exp\left({-\frac{x}{n^2}}\right) =\frac{\Gamma(\frac{k}{2})}{x^\frac{k}{2}}\sum_{n=1}^{\infty}\frac{\mu(n)}{n}{}_1F_1 \left(\frac{k}{2}; \frac{1}{2}; - \frac{\pi^2}{n^2x} \right) + \frac{1}{2}\sum_{\rho}\frac{\Gamma(\frac{k-\rho}{2})}{\zeta'(\rho)}x^{-\frac{(k-\rho)}{2}},
\end{align}
where the sum over $\rho$ runs through the non-trivial zeros of  $\zeta(s)$. }
 
Inspired from the above identity,  we \cite{AGM} also showed that the Riemann hypothesis is equivalent to the bound
\begin{equation}\label{AGM_bound}
P_k(x):= \sum_{n=1}^{\infty} \frac{\mu(n)}{n^k} \exp\left({-\frac{x}{n^2}}\right) = O_{\epsilon, k} \bigg(x^{-\frac{k}{2}+\frac{1}{4} + \epsilon }\bigg), \quad \mathrm{as}\,\, x \rightarrow \infty,
\end{equation}
for any positive $\epsilon$.  One can clearly observe that the above bound \eqref{AGM_bound} naturally generalizes the bound \eqref{Riesz type_Hardy_Littlewood} given by Hardy and Littlewood as well as the bound \eqref{Riesz} due to Riesz.  
 
Now we mention a character analogue of \eqref{Rama_Hardy_Little} that was obtained by Dixit \cite[Theorem 1.9]{dixit12}.  
Mainly,  he proved the following identity. 
\begin{theorem}\label{Dixit_character}
Let $\alpha$ and $\beta$ be two positive real numbers with $\alpha \beta =1$.  Let $\chi$ be an odd primitive Dirichlet character modulo $q$.  We have 
\begin{align}\label{odd_character_dixit}
\alpha \sqrt{\alpha} \sqrt{G(\chi)} \bigg(\sum_{n=1}^{\infty} \frac{\chi(n) \mu(n)}{n^2} \exp \left({-\frac{\pi \alpha^2}{q n^2}}\right) - \frac{q}{4 \pi \alpha^2} \sum_{\rho} \frac{\Gamma(\frac{2 - \rho}{2})}{L'(\rho , \chi)} \bigg(\frac{\pi}{q}\bigg)^{\frac{\rho}{2}} \alpha^{\rho} \bigg)  \nonumber \\  
= \beta \sqrt{\beta} \sqrt{G(\chi)} \bigg(\sum_{n=1}^{\infty} \frac{\chi(n) \mu(n)}{n^2} \exp \left({-\frac{\pi \beta^2}{q n^2}}\right) - \frac{q}{4 \pi \beta^2} \sum_{\rho} \frac{\Gamma(\frac{2 - \rho}{2})}{L'(\rho , \overline{\chi})} \bigg(\frac{\pi}{q}\bigg)^{\frac{\rho}{2}} \beta^{\rho} \bigg),
\end{align}
where $\rho$ runs over the non-trivial zeros of $L(s, \chi)$ and $L(s, \overline{\chi})$ are assumed to be simple. 
\end{theorem}
Here we remark that the convergence of the  above two infinite series over all the non-trivial zeros of $L(s, \chi)$ and $L(s,  \bar{\chi})$ can be shown under the similar bracketing condition \eqref{bracketing}.  Corresponding to an even character $\chi$,  an identity similar to \eqref{odd_character_dixit} was also obtained by Dixit \cite[Theorem 1.10]{dixit12}. 
 
 One of the major goals of the current paper is to establish a character analogue of the identity \eqref{AGM} and derive a one-variable generalization of the above identity \eqref{odd_character_dixit} of Dixit.   Furthermore,  the generalized character analogue identity motivates us to derive  equivalent criteria for the generalized Riemann hypothesis for the Dirichlet $L$-function.  In the next section,  we record all the main results of the present paper.

 \section{Main results}
 
 \begin{theorem}\label{character_analogue_AGM} 
 Let $\chi$ be a  primitive character modulo $q$ and $k \geq 1$ be any real number. 
 Assume that all the non-trivial zeroes of $L(s,\chi)$ are simple.  Then,  for any positive real number $x$,  we have
\begin{align}
      \sum_{n=1}^{\infty} \frac{\chi(n) \mu(n)}{n^k} \exp \left({-\frac{\pi x^2}{q n^2}}\right) & =  \frac{i^a \sqrt{q}}{G(\chi)} \bigg(\frac{q}{\pi x^2}\bigg)^{\frac{k+a}{2}} \bigg( \frac{\pi}{q} \bigg)^{a+\frac{1}{2}} \frac{\Gamma(\frac{k+a}{2})}{\Gamma(a+\frac{1}{2})} \nonumber \\
     & \times  \sum_{n=1}^{\infty} \frac{\overline{\chi(n)} \mu(n)}{n^{1+a}} {}_1F_{1} \left( \frac{k+a}{2}; a+\frac{1}{2}; - \frac{\pi}{qn^2 x^2} \right) \nonumber \\
     &+ \frac{1}{2}\sum_{ \rho } \frac{\Gamma(\frac{k-\rho}{2})}{ L'(\rho , \chi)} \bigg(\frac{q}{\pi x^2}\bigg)^{\frac{k-\rho}{2}},  \label{character_analogue}
\end{align}
where $\rho$ runs over the non-trivial zeroes of $L(s,\chi)$,  which satisfy the bracketing conditions, that is,   the terms corresponding to the non-trivial zeros $\rho_1$ and $\rho_2$ will be inside the same bracket if 
\begin{align}\label{bracketing}
|\Im(\rho_1) - \Im(\rho_2)| < \exp \left( -\frac{C \Im(\rho_1)}{\log(\Im(\rho_1) )} \right) + \exp \left( -\frac{C \Im(\rho_2)}{\log(\Im(\rho_2) )} \right)
\end{align}
 for some positive constant $C$. 
\end{theorem}
\begin{remark}
The above theorem can be thought of as a character analogue of \eqref{AGM}.  Substituting $q=1$ i.e., considering $\chi$ to be the trivial character and replacing $x$ by $x/\sqrt{\pi}$ in \eqref{character_analogue},  one can  immediately obtain \eqref{AGM}.   
\end{remark}
 \begin{remark}
 Theorem \ref{character_analogue_AGM} is a one-variable generalization of Theorem \ref{Dixit_character}.  At a glance,  they may look totally different from each other,  but substituting $k=2$ and considering $\chi$ be an odd character in \eqref{character_analogue} and doing a little more simplification,  we can derive Theorem \ref{Dixit_character}.  
 \end{remark}
 
The bound \eqref{AGM_bound},  which gives equivalent criteria for the Riemann hypothesis,   motivates us to expect that,  for any $\epsilon >0$,  
\begin{equation}\label{GM_bound}
P_{k, \chi}(x):=  \sum_{n=1}^{\infty} \frac{\chi(n) \mu(n)}{n^k} \exp \left({-\frac{ x}{n^2}}\right) = O_{\epsilon,  k} \bigg(x^{-\frac{k}{2}+\frac{1}{4} + \epsilon }\bigg), \quad \mathrm{as}\,\, x \rightarrow \infty,
\end{equation}
must be the equivalent criteria for the generalized Riemann hypothesis.  More generally,  we obtain the following equivalent criteria.  
 
 \begin{theorem}\label{Garg_Maji bound} Let $k \geq 1$ and $\ell$ be positive real numbers.  The  generalized Riemann hypothesis is equivalent to the bound
\begin{equation}\label{Two_variable_GM_bound}
P_{k,  \ell,  \chi}(x):=  \sum_{n=1}^{\infty} \frac{\chi(n) \mu(n)}{n^{k}} \exp \left(- \frac{ x}{n^{\ell}}\right) = O_{\epsilon,k,\ell} \bigg(x^{-\frac{k}{\ell}+\frac{1}{2 \ell} + \epsilon }\bigg), \quad \mathrm{as}\,\, x \rightarrow \infty. 
\end{equation}
 \end{theorem}
 Letting $\ell =2$ in \eqref{Two_variable_GM_bound},  one can immediately see the bound \eqref{GM_bound}.  In the next section,  we collect all the necessary results which are essential to prove our main results.  
 
 \section{Nuts and Bolts}
 Let $\chi$ be a primitive Dirichlet character modulo $q$ and $L(s,  \chi)$ be the Dirichlet $L$-function associated to $\chi$.   One of the main motivations of Dirichlet was to show the existence of infinitely many primes in any arithmetic progression with the help of the theory of $L(s,  \chi)$.  Similar to the Riemann zeta function,  $L(s, \chi)$ obeys the following symmetric functional equation:
\begin{equation}\label{functional_Dirichlet-L}
\Big(\frac{q}{\pi} \Big)^{\frac{a + s}{2}}\Gamma\left(\frac{a+ s}{2}\right) L(s, \chi)= \frac{G(\chi)}{i^a \sqrt{q}} \Big(\frac{q}{\pi} \Big)^{\frac{1 - s + a}{2}}  \Gamma\left(\frac{1-s+a}{2}\right) L(1-s, \overline{\chi}),
\end{equation}
where $G(\chi)$ denotes the Gauss sum and the constant $a$ is defined as
\begin{align}\label{defn_a}
a:= \begin{cases}
0,  & \rm{if}\,\,  \chi \,\, \textrm{is even},  \\
1,  & \rm{if}\,\,  \chi \,\, \textrm{is odd}. 
\end{cases}
\end{align}
\begin{lemma}\label{Derivative_L}
Let  $\chi$ be any primitive character of conductor $q$  and $m \geq 2$ be an integer such that $\chi(-1) = (-1)^m$,  then
\begin{align}
    \frac{(m-2)!\, q^{m-2} G(\chi)}{2^{m-1} \pi^{m-2}i^{m-2}} L(m-1, \bar{\chi}) = L'(2-m, \chi).
\end{align}
\end{lemma}
\begin{proof}
One can find the proof of this lemma in  \cite[Lemma 3.1]{ABYZ}.  
\end{proof}

\begin{lemma}\label{bound_1by_L(s,chi)}
Suppose there exists a sequence of arbitrarily large positive numbers $T$ with $|T-\Im(\rho)| >
 \exp\left(-\frac{A_0 \Im(\rho) }{\log(\Im(\rho))}\right)$,  for every non-trivial zero of $L(s, \chi)$. 
 Then,
 \begin{equation*}
  \frac{1}{|L(\sigma + i T)|} < e^{A_1 T},
 \end{equation*}
 for some constant $0<A_1 < \pi/4$.
 \end{lemma}
 \begin{proof}
We can obtain a proof of this result along the lines given in \cite[p.~219]{Tit}.  
 \end{proof}

To obtain the equivalent criteria \eqref{Riesz type_Hardy_Littlewood} for the Riemann hypothesis,   Hardy and Littlewood used the following crucial identity due to Riesz,  namely, 
\begin{align*}
\int_{0}^{\infty} x^{-s-1} P_1(x) {\rm d}x = \frac{\Gamma(-s)}{\zeta(2s+1)},
\end{align*}
valid in the region $0 <\Re(s) < 1$.  
The next lemma gives us a two-variable generalization of this identity.  
 \begin{lemma}\label{Mellin_P_two variable}
 Let $\chi$ be a Dirichlet character.  Let $k \geq 1$ and $\ell > 0$ be  positive real numbers. 
In the region $\frac{1-k}{\ell} <\Re(s) < 1$, except at $s=0$,  we have
\begin{align*}
\int_{0}^{\infty} x^{-s-1} P_{k,  \ell , \chi}(x) {\rm d}x = \frac{\Gamma(-s)}{L(\ell s+k, \chi)}.  
\end{align*}
\end{lemma}

\begin{proof}
Making use of the definition of $P_{k,  \ell , \chi}(x)$,  we see that
\begin{align}\label{another form_P_k(x)}
P_{k,  \ell , \chi}(x) = \sum_{n=1}^{\infty} \frac{\chi(n) \mu(n)}{n^{k}} \exp \left({-\frac{x}{n^{\ell}}}\right) & = \sum_{n=1}^{\infty} \frac{\chi(n) \mu(n)}{n^{k}} \sum_{m=0}^\infty \frac{(-1)^m  x^{m} }{m! n^{\ell m} } \nonumber \\
& = \sum_{m=0}^\infty \frac{(-1)^m  x^{m} }{m!  L(k + \ell m, \chi) } .
\end{align}
For $\Re(s) > \frac{1-k}{ \ell}$,  utilizing the series representation of $L(\ell s +k,  \chi)$,  one can  write
\begin{align}\label{Mellin transform of P_k(x)}
L(\ell s+k, \chi) \int_{0}^{\infty} x^{-s-1} P_{k, \ell, \chi}(x) {\rm d}x  & = \sum_{n=1}^{\infty} \frac{\chi(n) \mu(n) }{n^{k}} \int_{0}^{\infty}  \frac{x^{-s-1}}{n^{\ell s}} P_{k, \ell, \chi}(x) {\rm d}x  \nonumber \\
& = \sum_{n=1}^{\infty} \frac{\chi(n) \mu(n) }{n^{k}}  \int_{0}^{\infty}  x^{-s-1} P_{k, \ell, \chi}\left(\frac{x}{n^{\ell}} \right) {\rm d}x \nonumber \\
& =  \int_{0}^{\infty}  x^{-s-1} \sum_{n=1}^{\infty} \frac{\chi(n) \mu(n) }{n^{k}} P_{k, \ell, \chi}\left(\frac{x}{n^{\ell}} \right) {\rm d}x.
\end{align}
In the penultimate step,  we replaced $x$ by $x/n^{\ell}$.  
Now we shall try to simplify the infinite series present in the above equation.  
Employ \eqref{another form_P_k(x)} to see that
\begin{align}\label{exp(-x)}
\sum_{n=1}^{\infty} \frac{\chi(n)}{n^{k}} P_{k , \ell, \chi}\left(\frac{x}{n^{\ell}} \right)  =  \sum_{n=1}^{\infty} \frac{\chi(n)}{n^{k}}   \sum_{m=0}^\infty \frac{(-1)^m  x^{m} }{m! n^{\ell m}  L(k+\ell m, \chi) } 
 =\sum_{m=0}^\infty \frac{(-1)^m  x^{m} }{m!} = e^{-x}.
\end{align}
At this situation,  combining \eqref{exp(-x)} and \eqref{Mellin transform of P_k(x)},  we arrive at 
\begin{align*}
L(\ell s+k, \chi) \int_{0}^{\infty} x^{-s-1} P_{k , \ell,  \chi}(x) {\rm d}x =  \int_{0}^{\infty}  x^{-s-1} e^{- x} {\rm d} x  =  \Gamma(-s),
\end{align*}
for $\Re(s)<0$.  Therefore,  analytic continuation of $\Gamma(-s)$ implies that the above identity can  be further extended to the right half plane except  at $s=0$ and positive integers. 

\end{proof}

Next, we  record Euler's summation formula,  which will play a crucial role for finding the equivalent criteria for the generalized Riemann hypothesis.  
 \begin{lemma} \label{Euler's summation}
 Let $\{ a_n\}$ be a sequence of complex numbers and $f(t)$ be a continuously differentiable function on $[1,x]$.  Consider $A(x):= \sum_{1 \leq n \leq x} a_n$.  Then we have
 \begin{align*}
 \sum_{ 1\leq n \leq x} a_n f(n) = A(x) f(x) - \int_{1}^{x} A(t) f'(t) {\rm d}t.
 \end{align*}
 \end{lemma}
\begin{proof}
Proof of this result can be found in \cite[p.~17]{Murty}
\end{proof}

Here we state the Meijer $G$-function \cite[p.~415, Definition 16.17]{NIST},  an important special function.  In particular,  it reduces to many well-known special functions.  
Let $m,n,p,q$ be non-negative integers such that $0\leq m \leq q$, $0\leq n \leq p$.  Let $a_1, \cdots, a_p$ and $b_1, \cdots, b_q$ be complex numbers along with $a_i - b_j \not\in \mathbb{N}$ for $1 \leq i \leq n$ and $1 \leq j \leq m$.  The Meijer $G$-function is defined by the following line integral: 
\begin{align}\label{Meijer-G}
G_{p,q}^{\,m,n} \!\left(  \,\begin{matrix} a_1,\cdots , a_p \\ b_1, \cdots , b_q \end{matrix} \; \Big| z   \right) := \frac{1}{2 \pi i} \int_L \frac{\prod_{j=1}^m \Gamma(b_j - s) \prod_{j=1}^n \Gamma(1 - a_j +s) z^s  } {\prod_{j=m+1}^q \Gamma(1 - b_j + s) \prod_{j=n+1}^p \Gamma(a_j - s)}\mathrm{d}s,
\end{align}
where the line of integration $L$ goes from $-i \infty$ to $+i \infty$, which  separates the poles of the factors $\Gamma(1-a_j+s)$  from those of the factors $\Gamma(b_j-s)$.   The above integral converges if we have $p+q < 2(m+n)$ and $|\arg(z)| < \left(m+n - \frac{p+q}{2} \right) \pi$. 

Next,  we write Slater's theorem \cite[p.~415, Equation 16.17.2]{NIST}, which tells us that the Meijer $G$-function can be written in terms of generalized hypergeometric functions. 
If $p \leq q$ and $ b_j - b_k \not\in \mathbb{Z}$ for $j\neq k$, $1 \leq j, k \leq m$, then 
\begin{align}\label{Slater}
& G_{p,q}^{\,m,n} \!\left(   \,\begin{matrix} a_1, \cdots , a_p \\ b_1, \cdots , b_q \end{matrix} \; \Big| z   \right)  \\
& \quad = \sum_{k=1}^{m} A_{p,q,k}^{m,n}(z) {}_p F_{q-1} \left(  \begin{matrix}
1+b_k - a_1,\cdots, 1+ b_k - a_p \\
1+ b_k - b_1, \cdots, *, \cdots, 1 + b_k - b_q 
\end{matrix} \Big| (-1)^{p-m-n} z  \right), \nonumber 
\end{align}
where $*$ means that the entry $1 + b_k - b_k$ is removed and 
\begin{align*}
A_{p,q,k}^{m,n}(z) := \frac{ z^{b_k}  \prod_{ j=1,  j\neq k}^{m} \Gamma(b_j - b_k ) \prod_{j=1}^n   \Gamma( 1 + b_k -a_j ) }{ \prod_{j=m+1}^{q} \Gamma(1 + b_k - b_{j}) \prod_{j=n+1}^{p} \Gamma(a_{j} - b_k)  }.
\end{align*}

 \section{Proof of the main results}\label{proof_main results}
 
 \begin{proof}[Theorem \ref{character_analogue_AGM}][]
The inverse Mellin transform of $\Gamma(s)$ suggests that,   for $-1 < c <0$,  one can obtain
 \begin{align}\label{inverse_Mellin}
e^{-x} - 1 = \frac{1}{2 \pi i}  \int_{c-i\infty}^{c+i \infty}\Gamma(s) x^{-s} {\rm d}s.  
\end{align}
Let $L(s,  \chi)$ be the Dirichlet $L$-function associated to the primitive Dirichlet character $\chi$ modulo $q$.  Given that $x$ is a positive real number and letting $k \geq 1$ be any real number,  we can write by making use of \eqref{inverse_Mellin},
\begin{align}
     \sum_{n=1}^{\infty} \frac{\chi(n) \mu(n)}{n^k} \exp \left({-\frac{\pi x^2}{q n^2}}\right)&=\sum_{n=1}^{\infty} \frac{\chi(n) \mu(n)}{n^k} + \sum_{n=1}^{\infty} \frac{\chi(n) \mu(n)}{n^k} \left(\exp\left({-\frac{\pi x^2}{q n^2}}\right) -1 \right) \nonumber \\
     &= \frac{1}{L(k,  \chi) }+\sum_{n=1}^{\infty} \frac{\chi(n) \mu(n)}{n^k}\frac{1}{2\pi i} \int_{c-i\infty}^{c+i \infty}\Gamma(s) \left(\frac{\pi x^2}{q n^2}\right)^{-s} {\rm d}s  \nonumber \\
     & =  \frac{1}{L(k,  \chi) } +\frac{1}{2\pi i}\int_{c-i \infty}^{c+i \infty}\frac{\Gamma(s)}{L(k-2s, \chi)} \left(\frac{\pi x^2}{q}\right)^{-s} {\rm d}. \label{main equation}
\end{align}
 Here,  change in the order of summation and integration in the final step, is valid due to absolute convergence of the infinite series,  since $\Re(k-2s) >1$.  We must try to analyze the following vertical line integral: 
\begin{align}\label{main_line_integration}
V_{k, \chi}^{(1)}(x):= \frac{1}{2\pi i}\int_{c-i \infty}^{c+i \infty} \frac{\Gamma(s)}{L(k-2s, \chi)} \left(\frac{\pi x^2}{q}\right)^{-s}  {\rm d}s. 
\end{align}
At first,  we shall find the poles of the integrand.  We know that $\Gamma(s)$ has simple poles at non-positive integers.  Again,  trivial zeros of $L(k-2s,  \chi)$ will also contribute to the simple poles of the integrand.  Employing the functional equation of $L(s,  \chi)$,  one can detect that the trivial zeros of $L(k-2s,  \chi)$ are at $\frac{k+a+2n}{2}$ for any $n \in \mathbb{N}\cup \{ 0\}$,  where the constant $a$ is defined as in \eqref{defn_a}.  Another crucial observation is that there will be infinitely poles of  the integrand due to the non-trivial zeros of $L(k-2s, \chi)$  in the critical strip $\frac{k - 1}{2}  < \Re(s) < \frac{k}{2}$.  As we are interested to evaluate the contributions of the poles corresponding to these non-trivial zeros of $L(k-2s, \chi)$,   we must shift the line of integration $\Re(s)=c$ to $\Re(s)= d$ with $ \frac{k+1}{2} < d < \frac{k}{2} + 1$.  Mainly,  we construct a rectangular contour $\mathfrak{C}$ with corner points $c\pm i T,  d \pm i T$,  where $T$ is some large positive real number.  Now,  making use of Cauchy's residue theorem,  we see that
\begin{align}\label{CRT_application}
    \frac{1}{2\pi i}\int_{\mathfrak{C}}\frac{  \Gamma(s)  }{L(k-2s, \chi) }  \left(\frac{\pi x^2}{q}\right)^{-s} {\rm d}s = R_{0} + R_{\frac{k+a}{2}} + \mathcal{R}_{T, \chi}(x),
\end{align}
 where $R_0$ and $R_{\frac{k+a}{2}}$ denote the residual terms corresponding to $s=0$ and $s= \frac{k+a}{2}$ respectively,  and the residual term $\mathcal{R}_{T, \chi}(x)$ corresponds to the non-trivial zeros $\rho$ of $L(k-2s,  \chi)$ with $|\Im(\rho)| < T$.  
 Without much effort,  we can evaluate residual terms $R_0$ and $R_{\frac{k+a}{2}}$.  One can find out that
 \begin{align}\label{Residue_first_second}
 R_0  = \frac{1}{L(k, \chi)},  \quad
  R_{\frac{k+a}{2}}  = - \frac{\Gamma\left( \frac{k+a}{2} \right) \left( \frac{\pi x^2}{q}  \right)^{-\frac{k+a}{2}}}{2 L'(-a, \chi)}.  
 \end{align}
Again,  if we assume that all the non-trivial zeros of $L(s, \chi)$ are simple,  then we have
\begin{align}\label{infinite_residual term}
\mathcal{R}_{T, \chi}(x)= \sum_{ |\Im(\rho)| < T }\, \lim_{ s \rightarrow \frac{k-\rho}{2}} \frac{ \left( s-  \frac{k-\rho}{2} \right) \Gamma(s)   }{L(k-2s, \chi)  } \left(\frac{\pi x^2}{q}\right)^{-s} = -\frac{1}{2}  \sum_{ | \Im( \rho)|<T } \frac{\Gamma(\frac{k-\rho}{2})}{ L'(\rho , \chi)} \left(\frac{\pi x^2}{q}\right)^{-\frac{k-\rho}{2}},
\end{align}
where the sum runs through all the non-trivial zeros of $L(s, \chi)$.  Now one of our main goals is to show that the following horizontal integrals 
\begin{align*}
H_{T, \chi}^{(1)}(x) & :=\frac{1}{2\pi i}\int_{d+ iT }^{c+i T} \frac{\Gamma(s)}{L(k-2s, \chi)} \left(\frac{\pi x^2}{q}\right)^{-s}  {\rm d}s,  \\
  H_{T, \chi}^{(2)}(x) & := \frac{1}{2\pi i}\int_{c-i T}^{d - i T} \frac{\Gamma(s)}{L(k-2s, \chi)} \left(\frac{\pi x^2}{q}\right)^{-s}  {\rm d}s.
\end{align*}
will  vanish as $T \rightarrow \infty$.  Now,  we shall make use of Stirling's formula for $\Gamma(s)$,   that is,  for $c \leq \sigma \leq d$,  
\begin{align*}
|\Gamma(\sigma + i T) \ll |T|^{\sigma - \frac{1}{2}} e^{- \frac{\pi}{2}|T|},  \quad \textrm{as} \,\, |T| \rightarrow \infty, 
\end{align*}
 and together with the bound for $1/L(s, \chi)$, i.e.,  using Lemma \ref{bound_1by_L(s,chi)},  one can see that
\begin{align}
| H_{T, \chi}^{(1)}(x)  | \ll |T|^{\sigma - \frac{1}{2}} e^{2 A_1 T- \frac{\pi}{2}|T|} ,
\end{align}
where $0 < A_1< \pi/4$.  This immediately implies that the horizontal integral vanishes as $T\rightarrow \infty$.  Similarly,  one can show that the same holds for the other horizontal integral  $H_{T, \chi}^{(2)}(x)$.  Thus  letting $T \rightarrow \infty$ in \eqref{CRT_application} and considering the fact that the horizontal integrals tend to zero and in view of \eqref{main equation} and \eqref{main_line_integration},  we reach 
\begin{align}\label{Using CRT}
     \sum_{n=1}^{\infty} \frac{\chi(n) \mu(n)}{n^k} \exp \left({-\frac{\pi x^2}{q n^2}}\right)  = V_{k, \chi}^{(2)}(x) - R_{\frac{k+a}{2}} - \mathcal{R}_{\chi}(x),  
\end{align}
where the residual term $\mathcal{R}_{\chi}(x)$ contains the contribution of all the non-trivial zeros of $L(s, \chi)$,  that is, 
 \begin{align}\label{infinite_residual_term}
\mathcal{R}_{\chi}(x) = -\frac{1}{2}  \sum_{ \rho } \frac{\Gamma(\frac{k-\rho}{2})}{ L'(\rho , \chi)} \left(\frac{\pi x^2}{q}\right)^{-\frac{k-\rho}{2}},
\end{align}
 and the right vertical integral is denoted by 
 \begin{align}\label{right vertical integral}
 V_{k, \chi}^{(2)}(x) := \frac{1}{2\pi i}\int_{d - i\infty }^{d+i \infty} \frac{\Gamma(s)}{L(k-2s, \chi)} \left(\frac{\pi x^2}{q}\right)^{-s}  {\rm d}s,
 \end{align}
 where $ \frac{k+1}{2} < d < \frac{k}{2} + 1$.  At this juncture,  we shall try to simplify the above integral.  Here we make use of the functional equation \eqref{functional_Dirichlet-L} of $ L(s, \chi)$. 
Replace $s$ by $k-2s$ in \eqref{functional_Dirichlet-L} and simplify to see that
\begin{align}\label{Zeta (k-2s)}
    \frac{1}{ L(k-2s, \chi)} = \frac{i^a}{G(\chi)} \frac{\Big(\frac{q}{\pi} \Big)^{k-2s} \Gamma\left(\frac{k + a}{2} - s \right)}{\Gamma\left(\frac{1+a-k}{2} + s \right) L(1-k+2s, \overline{\chi})}.
\end{align}
Putting \eqref{Zeta (k-2s)} in \eqref{right vertical integral},  we obtain 
\begin{align}\label{second form_V(x,k)}
V_{k, \chi}^{(2)}(x) = \frac{1}{\epsilon(\chi)}  \left(\frac{q}{\pi}  \right)^{k- \frac{1}{2} }  \frac{1}{2 \pi i }   \int_{d - i \infty}^{d +i \infty}   \frac{\Gamma(s) \Gamma(\frac{k+a}{2} - s)}{\Gamma(\frac{1+a-k}{2}+s)} \frac{ \left( \frac{q x^2}{\pi}  \right)^{-s}}{L(1-k+2s, \overline{\chi})}{\rm d}s, 
\end{align}
where $\epsilon(\chi) = \frac{G(\chi)}{i^a \sqrt{q}}$.  At this point,  one can verify that the argument of $L(1-k+2s, \bar{\chi})$ lies in the interval $2<\Re(1-k+2s)<3 $ since $\frac{k+1}{2} < \Re(s)= d < \frac{k}{2}+1$.  Thus,  we can write 
\begin{align}\label{series_L}
\frac{1}{L(1-k+2s, \overline{\chi})} = \sum_{n=1}^{\infty} \frac{\mu(n) \overline{\chi(n)} }{n^{1-k+2s}}.
\end{align}
Substituting the expression \eqref{series_L} in  \eqref{second form_V(x,k)} and then interchanging the integration and summation,  equation \eqref{second form_V(x,k)} takes the shape 
 \begin{align}\label{thrid form_V(x,k)}
V_{k, \chi}^{(2)}(x)  = \frac{1}{\epsilon(\chi)}  \Big(\frac{q}{\pi} \Big)^{k-\frac{1}{2}}  \sum_{n=1}^\infty  \frac{ \mu(n) \overline{\chi(n)} }{ n^{1-k}}  J(X_{n,q}, k, a),
\end{align}
 where 
\begin{align}\label{I(X_n, k)}
J(X_{n,q}, k, a):=  \frac{1}{2 \pi i } \int_{d - i \infty}^{d +i \infty}   \frac{ \Gamma(s) \Gamma\left(\frac{k+a}{2}-s \right)}{\Gamma\left(\frac{1+a-k}{2} +s \right) }  X_{n,q}^{-s} {\rm d}s,
\end{align}
and $X_{n,q} := \frac{q\,  (n x)^2}{\pi}$.  At this moment,  our final attempt is to simplify the above integral and write it in terms of well known functions.  Here we shall use the definition \eqref{Meijer-G} of the Meijer $G$-function.  Evaluating the poles of $\Gamma(s)$ and $\Gamma\left(\frac{k+a}{2}-s \right)$,  one can verify that the line of integration $\Re(s)=d$ does not separate the poles of $\Gamma\left(\frac{k+a}{2}-s \right)$ from the poles of $\Gamma(s)$ since $ \frac{k+1}{2} < d < \frac{k}{2}+1$. 
So, we draw a new line of integration $\Re(s)=d_1$ with $0< d_1 < \frac{k}{2}$.  Now we can easily check that this new line  $\Re(s)=d_1$  does separate the poles of $\Gamma\left(\frac{k+a}{2}-s \right)$ from the poles of $\Gamma(s)$.  Once again,  considering a rectangular contour with corners $d\pm i T,  d_1 \pm i T$,  we  apply Cauchy's residue theorem.   Letting $T \rightarrow \infty$ and simplifying,  we show that 
\begin{align}
J(X_{n,q}, k, a) & = \frac{1}{2 \pi i } \int_{d_1 - i \infty}^{d_1 +i \infty}   \frac{ \Gamma(s) \Gamma\left(\frac{k+a}{2}-s \right)}{\Gamma\left(\frac{1+a-k}{2} +s \right) }  X_{n,q}^{-s} {\rm d}s + \mathrm{Res}_{{s=\frac{k+a}{2}}}  \frac{ \Gamma(s) \Gamma\left(\frac{k+a}{2}-s \right)}{\Gamma\left(\frac{1+a-k}{2} +s \right) }  X_{n,q}^{-s} \nonumber \\
& =\frac{1}{2 \pi i } \int_{d_1 - i \infty}^{d_1 +i \infty}   \frac{ \Gamma(s) \Gamma\left(\frac{k+a}{2}-s \right)}{\Gamma\left(\frac{1+a-k}{2} +s \right) }  X_{n,q}^{-s} {\rm d}s - \frac{ \Gamma\left(\frac{k+a}{2} \right)}{\Gamma\left(a + \frac{1}{2} \right) }  X_{n,q}^{-\frac{k+a}{2}}.  \label{another_right vertical} 
\end{align}
 As the line of integration $\Re(s)= d_1$ separates the poles of $\Gamma(s)$ from the poles of $\Gamma\left(\frac{k+a}{2}-s \right)$,  we use  the definition \eqref{Meijer-G} of the Meijer $G$-function,  with $m=n=p=1,  q=2$,  and $a_1=1,  b_1= \frac{k+a}{2},  b_2= \frac{1-a+k}{2}$.   Thus,  we can write
\begin{align}\label{in terms_Meijer-G}
\frac{1}{2 \pi i}   \int_{d_1 - i \infty}^{d_1 + i \infty}    \frac{ \Gamma(s) \Gamma\left(\frac{k+a}{2}-s \right)}{\Gamma\left(\frac{1+a-k}{2} +s \right) }  X_{n,q}^{-s} {\rm d}s =  G_{1,2}^{1,1} \left(\begin{matrix} 1 \\ \frac{k+a}{2},\frac{1-a+k}{2} \end{matrix} \Big| \frac{1}{X_{n,q}}\right).
\end{align}
 Note that $b_1 - b_2 = a - \frac{1}{2} \not\in \mathbb{Z}$,  so we can employ Slater's theorem \eqref{Slater}.  Therefore,  using Slater's theorem \eqref{Slater},  we get 
 \begin{align}\label{MeijerG_1F1}
G_{1,2}^{1,1} \left(\begin{matrix} 1 \\ \frac{k+a}{2},\frac{1-a+k}{2} \end{matrix} \Big| \frac{1}{X_{n,q}}\right) = \frac{\Gamma\left( \frac{k+a}{2} \right)}{X_{n,q}^{ \frac{k+a}{2}} \Gamma(a+\frac{1}{2})} {}_1F_{1} \left( \frac{k+a}{2}; a+\frac{1}{2}; - \frac{1}{X_{n,q}} \right).
\end{align}
Substituting \eqref{MeijerG_1F1} in \eqref{in terms_Meijer-G} and in view of \eqref{another_right vertical}, \eqref{I(X_n, k)} and \eqref{thrid form_V(x,k)},  one can see that the final expression of the right vertical integral becomes 
\begin{align*}
V_{k, \chi}^{(2)}(x) & =   \frac{1}{\epsilon(\chi)} \Big(\frac{q}{\pi} \Big)^{k- \frac{1}{2}}  \frac{\Gamma\left( \frac{k+a}{2} \right)}{\Gamma(a+\frac{1}{2})}  \sum_{n=1}^\infty  \frac{ \mu(n) \overline{\chi(n)}}{ n^{1-k} X_{n,q}^{ \frac{k+a}{2}} } \Bigg( {}_1F_{1} \left( \frac{k+a}{2}; a+\frac{1}{2}; - \frac{1}{X_{n,q}} \right) -1  \Bigg) \\
& = \frac{i^a \sqrt{q}}{G(\chi) x^{k+a}} \Big(\frac{q}{\pi} \Big)^{\frac{k-1-a}{2}}  \frac{\Gamma\left( \frac{k+a}{2} \right)}{\Gamma(a+\frac{1}{2})}  \sum_{n=1}^\infty  \frac{ \mu(n) \overline{\chi(n)}}{ n^{1+a}} \Bigg( {}_1F_{1} \left( \frac{k+a}{2}; a+\frac{1}{2}; - \frac{\pi}{q (nx)^2 } \right) -1  \Bigg),
\end{align*}
where in the last step we have substituted $\epsilon(\chi)= \frac{G(\chi)}{i^a \sqrt{q}}$ and  $X_{n, q}= \frac{q(nx)^2}{\pi}$.  
Now we employ Lemma \ref{Derivative_L} to simplify further.  Replace $m$ by $a+2$ in Lemma \ref{Derivative_L} to get
\begin{align}
    \frac{(a)! q^{a} G(\chi)}{2^{a+1} \pi^{a}i^{a}} L(1+a, \bar{\chi}) = L'(-a, \chi).
\end{align}
Substituting this expression and simplifying, we reach 
\begin{align}
V_{k, \chi}^{(2)}(x) &  = \frac{i^a \sqrt{q}}{G(\chi)} \bigg(\frac{q}{\pi x^2}\bigg)^{\frac{k+a}{2}} \bigg( \frac{\pi}{q} \bigg)^{a+\frac{1}{2}} \frac{\Gamma(\frac{k+a}{2})}{\Gamma(a+\frac{1}{2})}  \sum_{n=1}^{\infty} \frac{\overline{\chi(n)} \mu(n)}{n^{1+a}} {}_1F_{1} \left( \frac{k+a}{2}; a+\frac{1}{2}; - \frac{\pi}{qn^2 x^2} \right)  \nonumber \\
&  \hspace{8cm} -  \frac{\Gamma\left( \frac{k+a}{2} \right) \left( \frac{\pi x^2}{q}  \right)^{-\frac{k+a}{2}}}{2 L'(-a, \chi)}. \label{Final_2nd Vertical}
\end{align}
The last term in the above expression will be cancelled with the residual term $R_{\frac{k+a}{2}}$ present in \eqref{Residue_first_second}.  Thus,  finally substituting \eqref{Final_2nd Vertical} in \eqref{Using CRT} and together with residual terms \eqref{Residue_first_second} and \eqref{infinite_residual_term},  we complete the proof of Theorem \ref{character_analogue_AGM}.

\end{proof}

\begin{proof}[Theorem {\rm \ref{Garg_Maji bound}}][] 

Firstly,  assuming the bound \eqref{GM_bound} for  $P_{k, \ell, \chi}(x)$,  we shall show that all the non-trivial zeros of $L(s, \chi)$ lie on $\Re(s)= 1/2$.  Invoking Lemma \ref{Mellin_P_two variable},  one has
\begin{align}\label{Lemma_Mellin_P_k(x)}
 L(\ell s +k, \chi) \int_{0}^{\infty} x^{-s-1} P_{k, \ell, \chi}(x) {\rm d}x =  \Gamma(-s), 
\end{align}
valid in the region $\frac{1-k}{\ell} <\Re(s) < 1$, except at $s=0$.  Now we shall attempt to extend the region of validity of the above identity \eqref{Lemma_Mellin_P_k(x)} on the left half plane,  mainly,  in the region $\Re(s)> \frac{1}{2 \ell}-\frac{k}{\ell}$.  We choose a large positive real number $R$,  and then write 
\begin{align*}
  L(\ell s+k, \chi) \left( \int_{0}^{R} +  \int_{R}^{\infty}  \right) x^{-s-1}  P_{k, \ell, \chi}(x) {\rm d}x = \Gamma(-s).  
\end{align*}
Utilizing the bound \eqref{GM_bound} for $P_{k, \ell, \chi}(x)$,  we can clearly see that the unbounded part is analytic in the domain $\Re(s)> \frac{1}{2 \ell}-\frac{k}{\ell}$,  whereas using a trivial bound for $P_{k, \ell, \chi}(x)$,  one can show that the finite part is analytic for $\Re(s)<0$.  Therefore,  in particular,  the identity \eqref{Lemma_Mellin_P_k(x)} is analytic in the strip $ \frac{1}{2\ell}-\frac{k}{\ell} <\Re(s) < \frac{1-k}{\ell}$.  We know that $\Gamma(s)$ never vanishes,  which indicates that 
$L(\ell s+k, \chi)$ has no zero in the strip $ \frac{1}{2\ell}-\frac{k}{\ell} <\Re(s) < \frac{1-k}{\ell}$.  This is equivalent to saying that $L(s, \chi)$ does not vanish in the strip $ \frac{1}{2}< \Re(s) <1$ and thus the functional equation of $L(s, \chi)$ suggests that it has no zero in the strip $0< \Re(s) <\frac{1}{2}$.  This proves that all the non-trivial zeros of $L(s, \chi)$ will lie on the critical line $\Re(s)=1/2$. 

Now we shall try to show the converse part, that is, we assume that the generalized Riemann hypothesis for $L(s, \chi)$ is true.  In \cite[Proposition 5.14]{IK},  we find that the generalized Riemann hypothesis is equivalent to the following bound,  that is,  for any $\epsilon >0$, 
\begin{align}\label{summatory_Mobius}
S(x):= \sum_{1 \leq n \leq x} \chi(n) \mu(n) = O_{\epsilon} \left(x^{\frac{1}{2}+\epsilon} \right). 
\end{align}
This can be proved along the lines given in Titchmarch \cite[370]{Tit}. 
Now we employ Euler's partial summation formula,  i.e.,  Lemma \ref{Euler's summation} with $a_n = \chi(n) \mu(n)$ and $f(t)= t^{-k}$,  to see that
\begin{align}\label{T(m,n)}
T(m;n) := \sum_{j=m}^{n} \frac{\chi(j) \mu(j)}{j^{k}}  & =  S(n) f(n)- S(m-1) f(m-1) - \int_{m-1}^n S(t) f'(t) {\rm d}t.
\end{align}
Substituting \eqref{summatory_Mobius} in \eqref{T(m,n)},  we can readily see that
\begin{align}\label{bound for T(m,n)}
T(m;n) = O_{\epsilon,k}\left(m^{\frac{1}{2}-k+\epsilon}\right), 
\end{align}
uniformly in $n$.  We are interested to find the bound for the following infinite series 
\begin{align*}
P_{k, \ell, \chi}(x)= \sum_{n=1}^{\infty} \frac{\chi(n) \mu(n)}{n^{k}} \exp\left({-\frac{ x}{ n^{\ell}}}\right).
\end{align*}
Now for the sake of simplicity,  we replace $x$ by $x^\ell$ and divide the sum into two parts,  that is,  
\begin{align}\label{P_k(x^2) interms of S_1 and S_2}
P_{k, \ell, \chi}(x^{\ell})   := Y_1(x^{\ell}) + Y_2(x^{\ell}), 
\end{align}
where 
\begin{align}
Y_1(x^{\ell}) := \sum_{n=1}^{m-1}  \frac{\chi(n) \mu(n)}{n^{k}}  \exp\left({-\frac{ x^{\ell}}{ n^{\ell}}}\right),   \quad
 Y_2(x^{\ell})=  \sum_{n=m}^{\infty}  \frac{\chi(n) \mu(n)}{n^{k}} \exp\left({-\frac{ x^{\ell}}{ n^{\ell}}}\right),  \nonumber 
 \end{align}
and $m = [x^{1-\epsilon}]+1$.  First,  we trivially bound $Y_1(x^{\ell})$.  By using $|\chi(n)\mu(n)| \leq 1$,  one has
\begin{align*}
|Y_1(x^{\ell})| & 
\leq  \sum_{n=1}^{m-1} \exp\left({-\frac{x^{\ell}}{ m^{\ell}}}\right). 
\end{align*} 
As $m = [x^{1-\epsilon}] + 1$,  it follows that 
\begin{align}\label{final bound of 
Y_1}
Y_1(x^{\ell}) = O \left( x^{1-\epsilon} \exp( - x^{\ell \epsilon} ) \right). 
\end{align}
Now we shall concentrate on the  evaluation of the bound for $Y_2(x^{\ell})$.  
 Utilizing the definition \eqref{T(m,n)} of $T(m;n)$,   we  derive that, for any integer $N > m$, 
\begin{align}\label{sum upto N}
\sum_{n=m}^{N} \frac{\chi(n) \mu(n)}{n^{k}} \exp\left({-\frac{ x^{\ell}}{ n^{\ell}}}\right) 
 & =  \sum_{n=m}^{N-1} T(m;n) \left[  \exp\left( {-\frac{ x^{\ell}}{ n^{\ell}}} \right)- \exp\left( -\frac{ x^{\ell}}{ (n+1)^{\ell}} \right) \right] \nonumber  \\
& + T(m; N) \exp \left( - \frac{ x^{\ell}}{ N^{\ell}} \right) .
\end{align}
Allowing $N \rightarrow \infty$ in the above equation and invoking the bound \eqref{bound for T(m,n)} for $T(m,n)$,  one can deduce that 
\begin{align}
Y_2(x^{\ell})= Z(x^{\ell}) + O_{\epsilon,k}\left(m^{\frac{1}{2}-k +\epsilon} \right),  \label{Y_2}
\end{align}
where 
\begin{align*}
Z(x^{\ell}) & :=  \sum_{n=m}^{\infty} T(m;n) \left[  \exp\left( {-\frac{ x^{\ell}}{ n^{\ell}}} \right)- \exp\left( -\frac{ x^{\ell}}{ (n+1)^{\ell}} \right) \right] \\
& = \sum_{n=m}^\infty T(m; n) \left[ F(n) - F(n+1)\right],  
\end{align*}
with $F(y):= \exp \left( -\frac{ x^{\ell}}{y^{\ell}} \right)$.  To simplify further,  we make use of mean value theorem.  We write $ F(n+1) - F(n)= F'(c_n)= \frac{ \ell \, x^\ell}{c_n^{\ell+1}} \exp\left(- \frac{x^\ell}{c_n^\ell}   \right) $ for some $n<c_n<n+1$.   Substituting this expression and in view of the bound \eqref{T(m,n)},  we can show that
\begin{align}
\big|Z(x^{\ell}) \big|  &  \ll_{\epsilon,  \ell} m^{\frac{1}{2}-k+\epsilon}  \sum_{n=m}^{\infty} \frac{ x^{\ell} }{c_n^{\ell + 1 }} \exp\left(- \frac{x^\ell}{c_n^\ell}   \right) \nonumber  \\
 & \ll_{\epsilon,  \ell}  m^{\frac{1}{2}-k+\epsilon}  \sum_{n=m}^\infty \frac{x^\ell}{n^{\ell + 1}} \nonumber \\
 & \ll_{\epsilon,  \ell}  m^{\frac{1}{2}-k+\epsilon} \frac{x^{\ell}}{m^\ell}  \ll_{\epsilon,  \ell} x^{\frac{1}{2}-k+\epsilon'},  \label{bound_Z}
\end{align}
as $m \sim x^{1-\epsilon}$.   At this moment,  plugging \eqref{bound_Z} in \eqref{Y_2},  we see that
\begin{align}\label{final bound of Y_2}
|Y_2(x^{\ell})| = O_{\epsilon}  \left( x^{\frac{1}{2}-k +\epsilon}  \right).
\end{align}
Finally,  substituting the bounds \eqref{final bound of Y_1} and \eqref{final bound of Y_2} for $Y_1(x^{\ell})$ and $Y_2(x^{\ell})$ in \eqref{P_k(x^2) interms of S_1 and S_2} and comparing,  we can conclude that 
\begin{align*}
P_{k, \ell, \chi}( x^{\ell}) = O_{\epsilon,k,\ell}  \left( x^{\frac{1}{2}-k +\epsilon}  \right).
\end{align*}
At the end, replace $x$ by $x^{\frac{1}{\ell}}$ to obtain \eqref{GM_bound}.  This completes the proof of Theorem \ref{Garg_Maji bound}.

\end{proof}

\section{Concluding remarks} 


The present paper is motivated from the identity \eqref{Hardy-Littlewood} of Hardy and Littlewood.  
In 2012,  Dixit \cite{dixit12} obtained the identity \eqref{odd_character_dixit},  which is  a character analogue of the identity \eqref{Hardy-Littlewood}.  In the present paper,  we have established a one-variable generalization of Dixit's identity,  namely,  Theorem \ref{character_analogue}.  This generalization also provides a character analogue of the identity \eqref{AGM}.  Motivated from the work of Riesz, Hardy and Littlewood,  we have established the following equivalent criteria for the generalized Riemann hypothesis for $L(s,  \chi)$: For $k\geq 1,  \ell >0$,  we have
\begin{equation}\label{Two_variable_GM_bound_Final}
  \sum_{n=1}^{\infty} \frac{\chi(n) \mu(n)}{n^{k}} \exp \left(- \frac{ x}{n^{\ell}}\right) = O_{\epsilon} \bigg(x^{-\frac{k}{\ell}+\frac{1}{2 \ell} + \epsilon }\bigg), \quad \mathrm{as}\,\, x \rightarrow \infty. 
\end{equation}
This bound generalizes the bounds given by Riesz,  and Hardy and Littlewood as well as our previous bound \eqref{AGM_bound}.   The above bound motivates us to predict a more general observation.  Let $L(f,s)$ be a ``nice'' $L$-function for which the grand Riemann hypothesis \cite[p.~113]{IK} predicts that the non-trivial zeros will lie on the critical line $\Re(s)= 1/2$.  Suppose we write $1/L(f,s) = \sum_{n=1}^\infty \frac{\mu_f(n)}{n^s}$,  which is valid in some right half plane.  In \cite[Proposition 5.14]{IK},  one can see that the grand Riemann hypothesis is equivalent to the following bound: For any $\epsilon>0$, 
\begin{align}
\sum_{ 1 \leq n \leq x} \mu_f(x) \ll_{\epsilon, f} x^{1/2 + \epsilon}.  
\end{align}
Note that exactly the same bound in case of the generalized Riemann hypothesis was crucial to obtain the above bound \eqref{Two_variable_GM_bound_Final}.  This motivates us to give the following conjecture. 
\begin{conjecture}
Let $k \geq 1$ and $\ell>0$ be two real numbers.  Let $L(f,s)$ be a ``nice'' $L$-function.  The grand Riemann hypothesis for $L(f,s)$ is equivalent to the bound
\begin{align}
  \sum_{n=1}^{\infty} \frac{\mu_f(n)}{n^{k}} \exp \left(- \frac{ x}{n^{\ell}}\right) = O_{\epsilon,  f, k, \ell} \bigg(x^{-\frac{k}{\ell}+\frac{1}{2 \ell} + \epsilon }\bigg), \quad \mathrm{as}\,\, x \rightarrow \infty. 
\end{align}

\end{conjecture}


{\bf Acknowledgement:} We would like to thank Prof.  Atul Dixit and Dr. Pramod Eyyunni for giving useful suggestions. 
The second author wants to thank SERB for the Start-Up Research Grant SRG/2020/000144.

\end{document}